 \documentclass{alteredlms}
\usepackage{geometry}   
\usepackage{graphicx}
\usepackage{amssymb}
\usepackage{epstopdf}
 \newtheorem{theorem}{Theorem}[section]
    \newtheorem{corollary}[theorem]{Corollary}
   \newtheorem{lemma}[theorem]{Lemma}
    \newtheorem{proposition}[theorem]{Proposition}  
    
     \newnumbered{conjecture}[theorem]{Conjecture}  
 \newnumbered{question}[theorem]{Question} 
\newnumbered{definition}[theorem]{Definition}
\newnumbered{remark}[theorem]{Remark}
\newnumbered{example}[theorem]{Example}
\newnumbered{examples}[theorem]{Examples}
\newnumbered{keyexample}[theorem]{Key example}

  \newcommand{\Z}{\ensuremath{{\mathbb{Z}}}}

  \newcommand{\G}{\Gamma}

\newcommand{\AG}{A_\G}            
\newcommand{\p}{^\perp}

\newcommand{\Jv}{J_{[v]}}
\newcommand{\Jw}{J_{[w]}}

\newcommand{\Lv}{L_{[v]}}
\newcommand{\Lw}{L_{[w]}}
\newcommand{\Lu}{L_{[u]}}
\newcommand{\Jy}{J_{[y]}}
\newcommand{\Ly}{L_{[y]}}
\newcommand{\Jz}{J_{[z]}}

\title[Automorphism groups of RAAGs]{Finiteness properties of automorphism groups of right-angled Artin groups}
\author{Ruth Charney and Karen Vogtmann}
\classno{20F36 (primary)}

\extraline{R. Charney was partially supported by NSF grant DMS 0705396.  K. Vogtmann was partially supported by NSF grant DMS 0705960}

\begin{document}
\maketitle

\begin{abstract} We study the algebraic structure of the outer automorphism group of a general right-angled Artin group.  We show that this group is virtually torsion-free and has finite virtual cohomological dimension.  This generalizes results proved in \cite{ChCrVo} for two-dimensional right-angled Artin groups.  
\end{abstract}

\section{Introduction}

Associated to a finite simplicial graph $\G$ is a right-angled Artin group $\AG$ whose generators are the vertices of $\G$ and whose relations are commutators between adjacent vertices. These groups range from free groups (when $\G$ has no edges) to free abelian groups (when $\G$ is a complete graph), hence their outer automorphism groups interpolate between $Out(F_n)$ and $GL_n(\Z)$.  The automorphism groups of free groups and free abelian groups have many properties in common. In particular, they are virtually torsion-free, have finite virtual cohomological dimension, and satisfy a Tits alternative (every subgroup is either virtually solvable or contains a free group of rank 2).  One is naturally led to ask whether the same is true for automorphism groups of right-angled Artin groups.

In \cite{ChCrVo} we began a study of the groups $Out(\AG)$  by analyzing the case when the defining graph $\G$ is connected and has no triangles.  We used both algebraic and geometric methods to establish cohomological finiteness results as well as to prove that the Tits alternative holds.  The key algebraic  tools we used were certain projection homomorphisms.  To define these homomorphisms, we introduced a partial ordering and an  equivalence relation on the vertices of $\G$.  For each maximal vertex $v$, the projection homomorphism $P_v$ was defined on a finite index subgroup of $Out(\AG)$ and took its image in $Out(A_{lk(v)})$ where $lk(v)$ denotes the link of $v$ in $\G$.  These projection homomorphisms were assembled into a single homomorphism $P$, and the kernel of $P$ was computed precisely.    

In this paper we show how similar projection homomorphisms can be defined for general right-angled Artin groups.  We then determine enough information about the kernel of the composite homomorphism to be able to construct an inductive proof that  $Out(\AG)$ is virtually torsion-free.  Thus the virtual cohomological dimension of $Out(\AG)$ is defined,   and we go on to  prove that it is finite for any finite simplicial graph $\G$.

\section{Join subgroups}\label{joins}
Let $\G$ be a connected, simplicial graph with vertex set $V$, and $\AG$ the associated  right-angled Artin group.  For two vertices $v,w \in V$, let $d(v,w)$ denote the distance from $v$ to $w$ in $\G$.  The  link $lk(v)$ is the subgraph spanned by vertices at distance one from $v$, and the  star $st(v)$ is the subgraph spanned by vertices at distance at most one from $v$.  Note that $v\in lk(w)$ if and only if $w\in lk(v)$, and $v\in st(w)$ if and only if $w\in st(v)$.

We first need to define an appropriate partial ordering  and equivalence relation on vertices in $\G$, and prove some of their basic properties.  
For $v,w\in V$, set $\,v \leq w\,$ if $lk(v) \subseteq st(w)$. Note that this may occur in three ways:
\begin{enumerate}
\item $v=w$,
\item $d(v,w)=1$ and $st(v) \subseteq st(w)$,
\item $d(v,w)=2$ and $lk(v) \subseteq lk(w)$.
\end{enumerate}

\begin{lemma}\label{one} If $u\leq v\leq w$ and $d(u,v)=1$, then $d(v,w)\leq 1$.
\end{lemma}
\begin{proof} Since $d(u,v)=1$, $u\in lk(v)\subseteq st(w)$.  Therefore $w\in st(u)$, i.e. either $w=u$ or $w\in lk(u)\subseteq st(v)$.  In either case $d(v,w)\leq 1$.
\end{proof}

\begin{lemma}\label{transitive}   The relation $\leq$ is  transitive on the vertices of $\G$.
\end{lemma}

\begin{proof} Suppose $u\leq v\leq w$.  We need to check that $u\leq w$.
If $d(u,v)=0$ or $2$, then $lk(u)\subseteq lk(v)\subseteq st(w)$, so this is immediate.
If $d(u,v)=1$, then by Lemma~\ref{one}, $d(v,w)\leq 1$.  Since we already know that $lk(v)\subset st(w)$, this gives $st(v)\subseteq st(w)$, so $lk(u)\subseteq st(v)\subseteq st(w)$.
\end{proof}

Now define a relation on vertices by $v \sim w$ if $v \leq w$ and $w \leq v$, or equivalently, if one of the following holds, 
\begin{enumerate}
\item $v=w$,
\item $d(v,w)=1$  and $st(v) = st(w)$,
\item $d(v,w)=2$ and $lk(v)= lk(w)$.
\end{enumerate}
It follows from Lemma~\ref{transitive} that this is an equivalence relation and that $\leq$ induces a partial ordering on the set of equivalence classes.  
We denote the equivalence class of $v$ by $[v]$.

\begin{lemma} All elements of $[v]$ are the same distance from each other.
\end{lemma}
\begin{proof}
If $x,y\in [v]$ and $d(x,y)$=1, then any other $z\in[v]$ satisfies $y\leq x\leq z$, so by Lemma~\ref{one}, $d(x,z)=1$.  
\end{proof}

Thus $[v]$ generates either a free abelian subgroup of $\AG$ or a non-abelian free subgroup.  Let $V_{ab}$ denote the set of vertices $v$ with $A_{[v]}$ abelian, and $V_{fr}$ the set of vertices with $A_{[v]}$ non-abelian.  

\begin{lemma}\label{chain}  Suppose $[v_1] < [v_2] < \dots < [v_k]$.  Then there exists $j$, $0 \leq j \leq k$, such that 
$d(v_{i}, v_{i+1})=2$ for all $i <j$ and $d(v_{i}, v_{i+1})=1$  for all $i \geq j$.  Moreover, for $i > j$, 
$v_i \in V_{ab}$.
\end{lemma}
\begin{proof}
This is immediate from Lemma~\ref{one}.
\end{proof}

It is not necessarily the case that $v_i$ is non-abelian for $i \leq j$, since an abelian $[v]$ can be less than a non-abelian $[w]$ if $[v]$ consists of just a single vertex.

\medskip
The projection homomorphisms in \cite{ChCrVo} mapped first to groups associated to certain {\it maximal join subgraphs} containing a given maximal vertex, then to groups associated to the link of  the vertex.  In the general case, the definitions of these subgraphs must be modified as follows.

For two simplicial graphs $\Theta_1$ and $\Theta_2$, let $\Theta_1 \ast \Theta_2$ denote their  join, that is, the graph formed by joining every vertex of $\Theta_1$ to every vertex of $\Theta_2$ by an edge.  Note that if $\Theta=\Theta_1 \ast \Theta_2$, then $A_{\Theta}= A_{\Theta_1} \times A_{\Theta_2}$.

For a vertex $v$ in $\G$, define
\begin{eqnarray*}
L_{[v]} = lk(v) \backslash [v] \\
J_{[v]} = L_{[v]} \ast [v]
\end{eqnarray*}
$\Jv$ is called the \emph{join associated to $[v]$.}  Note that it is always the case that $st(v) \subseteq \Jv$ and equality holds if and only if $v \in V_{ab}$.  Note also, that the relation $w \in \Lv$ is symmetric; that is, $w$ lies in $\Lv$ if and only if $v$ lies in $\Lw$.

The special subgroup generated by $\Jv$ is a direct product
\[
A_{\Jv} = A_{\Lv} \times A_{[v]} 
\] 
where the second factor is either free or free abelian.  

In order to understand the relation between various equivalence classes of vertices and their associated joins, it is useful to consider a graph $\G_0$ defined as follows.  The vertices of $\G_0$ are the equivalence classes $[v]$ which are maximal with respect to the partial order $\leq$.  Two vertices $[v]$ and $[w]$ are joined by an edge in $\G_0$ if and only if $v$ and $w$ are adjacent in $\G$.  It follows from the definition of the equivalence relation that adjacency is independent of choice of representative.  We remark that a similar graph, based on this same equivalence relation, is described in recent work of Duncan, Kazachkov, and Remeslennikov \cite{DuKaRe}.

\begin{lemma} \label{G_0}
  $\G_0$ is a connected graph and every $w \in V$ lies in the join $\Jv$ associated to some $[v]$ in $\G_0$.  
\end{lemma}

\begin{proof}   To prove connectivity of $\G_0$, let $[v]$ and $[v']$ be vertices of $\G_0$.  Let $v=v_0, v_1, \dots ,v_n=v'$ be an edgepath in $\G$.  We proceed by induction on $n$.  If $n=1$, then $[v],[v']$ are adjacent in $\G_0$.  Suppose $n >1$.  Choose $[w]$ such that $[w]$ is maximal and $[v_1] \leq [w]$.  Then $v$ and $v_2$ lie in $lk(v_1) \subset st(w)$.  So either $[v]=[w]$ or $[v]$ is adjacent to $[w]$ in $\G_0$, and there is an edgepath $w, v_2, \dots ,v_n$ in $\G$.  By induction, $[w]$ is connected to $[v']$ by a path in $\G_0$.

For the second statement of the lemma, note that if $w$ is in the link of $u$, then for any $v$ with $[u] \leq [v]$, $w \in lk(u) \subset st(v) \subset \Jv$.  
\end{proof}

For a subgraph $\Theta$ of $\G$, denote the normalizer, centralizer  and center of $A_\Theta$  respectively  by $N(\Theta)$, $C(\Theta)$ and $Z(\Theta)$.  It follows from a theorem of E.~Godelle 
\cite{God03} (see also \cite{ChCrVo},  Proposition 2.2), that 
\[
N(\Theta) =A_{\Theta\cup \Theta\p} \quad C(\Theta)=A_{\Theta\p}
\quad Z(\Theta) = A_{\Theta\cap \Theta\p}. 
\]
where $\Theta\p$ is the set of vertices commuting with all elements of $\Theta$.  We will be particularly interested in the case of $\Theta=\Jv$.

\begin{lemma}\label{centers}
  If $[v]$ is maximal, then the centralizer, center, and normalizer of $A_{\Jv}$ are given by
\[
C(\Jv) = Z(\Jv)= A_{[v]}  \textrm{ if $v \in V_{ab}$ and 
					\{1\}  if $v \in V_{fr}$}
\]
\[
N(\Jv) = A_{\Jv}.
\]
\end{lemma}

\begin{proof} 
Suppose $[v]$ is maximal and some $u$ commutes with all of $\Jv$.  Then $lk(v) \subseteq \Jv \subseteq st(u)$, so by maximality of $[v]$, $[u]=[v]$.  Thus $J_v\p$ is contained in $[v]$ and the lemma  follows from the descriptions above.
\end{proof}

\begin{lemma} \label{adjacent-joins}
 Suppose $[v]$ and $[w]$ are adjacent vertices in $\G_0$.   Let $J_{v,w}$ denote the intersection of $\Jv$ and $\Jw$.  Then
\[
 C(J_{v,w}) {}= Z (J_{v,w}) = Z([v]) \times Z([w]) \times Z(\Lv \cap \Lw)
 \]
 \[
N(J_{v,w}) {}= A_{J_{v,w}}.
\]
\end{lemma}

\begin{proof}  Since $[v]$ and $[w]$ are adjacent, $[v] \subset \Lw$ and $[w] \subset \Lv$.  Thus,  $J_{v,w}$ decomposes as a join,  $J_{v,w}= [v] \ast [w] \ast  (\Lv \cap \Lw)$. Any generator commuting with both $[v]$ and $[w]$ lies in $J_{v,w}$, so $J_{v,w}\p \subset J_{v,w}$.  The lemma now follows from the formulas above.
\end{proof}


\section{Projection homomorphisms}\label{projections}

In this section we describe the projection homomorphism $P$, which will be defined on a finite-index subgroup of $Out(\AG)$.
We begin by reviewing the work of M.~Laurence \cite{Lau95}.  Building on the work of 
H.~Servatius \cite{Ser89}, Laurence described a finite set of generators for $Aut(\AG)$  
as follows.
\begin{enumerate}
 \item \emph{Inner automorphisms} conjugate the entire group by some generator $v$.
  
  \item   \emph{Symmetries}  are induced by symmetries of $\G$ and permute the generators. 

   \item \emph{Inversions}  send a standard generator of $\AG$ to its inverse.

   \item   \emph{Transvections}  occur whenever $v \leq w$, and send 
  $v\mapsto vw$. 
   
  \item  \emph{Partial conjugations}  occur whenever removing the (closed) star of a vertex  $v$  disconnects $\G$. If this happens,  a partial conjugation conjugates all of the generators in one   component of $\G- st(v)$ by $v$.
\end{enumerate}

\begin{definition} The  subgroup  of  $Aut(\AG)$   generated by inner automorphisms, inversions, partial conjugations and transvections is called the {\it pure automorphism group\,} and is denoted $Aut^0(\AG)$.  The image of $Aut^0(\AG)$ in $Out(\AG)$ is the group of {\it pure outer automorphisms\,} and is denoted $Out^0(\AG)$.
\end{definition}

The subgroups $Aut^0(\AG)$ and $Out^0(\AG)$ are easily seen to be normal and of finite index in $Aut(\AG)$ and $Out(\AG)$ respectively.  We remark that if $\AG$ is a free group or free abelian group, then $Aut^0(\AG)=Aut(\AG)$.  


\begin{proposition}\label{representative}  Let $[v]$ be a maximal equivalence class.  Then any $\phi \in Out^0(\AG)$ has a representative  $\phi_v \in Aut^0(\AG)$ which preserves both  $A_{[v]}$ and  $A_{J_{[v]}}$. 
\end{proposition}

\begin{proof} 
 It suffices to check that the proposition holds for each of the generators of $Out^0(\AG)$. It is clear for inversions.
Represent $\phi$ by some automorphism $\hat \phi\in Aut^0(\AG)$. 

{\it Partial conjugations:}  Suppose $\hat\phi$ is a partial conjugation by a generator $w$.  If $w$ is not in $\Jv$, then $d(w,v)\ge 2$, so $st(w)\cap \Jv\subseteq \Lv$.  If $[v]$ is abelian, $\Jv-st(w)$ is clearly connected.  If $[v]$ is nonabelian, then maximality of $[v]$ implies that $st(w)\cap\Jv$ is not all of $\Lv$, so that  in this case  too $\Jv-st(w)$ is connected.  Thus $\hat\phi$ is either trivial on $A_{\Jv}$ or acts as conjugation by $w$ on all of $A_{\Jv}$.  Composing $\hat\phi$ with the inner automorphism associated to $w$ produces the desired $\phi_v$. 
 
 If $w \in \Jv = [v] \ast \Lv$, then conjugation by $w$ preserves the two factors $A_{[v]}$ and $A_{\Lv}$, hence the same is true for the partial conjugation $\hat\phi$, so we may take $\phi_v=\hat\phi$.  

{\it Transvections:}  Suppose $u \leq w$ and $\hat\phi$ is the transvection $u \mapsto uw$. If $u$ is not in $\Jv$ then $\hat\phi$ is the identity on $A_{\Jv}$.    If $u\in \Lv$, then $\hat\phi$ fixes $A_{[v]}$, and $v\in lk(u)\subseteq st(w)$, so $w\in st(v)\subseteq \Jv$ and $\hat\phi$ preserves   $A_{\Jv}$.  If $u\in[v]$, then maximality of $[v]$ implies $[v]=[u]=[w]$, so the transvection preserves both $A_{[v]}$ and $A_{\Jv}$.  In all cases we may take $\phi_v=\hat\phi$.  
\end{proof}

The representative $\phi_v$ described in Proposition~\ref{representative} is well defined up to conjugation by an element of the normalizer of $A_{\Jv}$.  In light of Lemma~\ref{centers}, the restriction of $\phi_v$ to $A_{\Jv}$ is well defined up to an inner automorphism of $A_{\Jv}$.  Moreover, since $\phi_v$ preserves $A_{[v]}$, it projects to an automorphism of $A_{\Lv} \cong A_{\Jv} / A_{[v]}$. This immediately gives the following corollary.

\begin{corollary}  For every maximal $[v]$ there is a restriction homomorphism
\[
R_{[v]} :  Out^0(\AG) \to Out(A_{\Jv})
\]
and a projection homomorphism
\[
P_{[v]} :  Out^0(\AG) \to Out(A_{\Lv})
\]
\end{corollary}

We now assemble these homomorphisms, one  for each maximal equivalence class  $[v]$, to obtain a restriction homomorphism  
\[
R= \prod R_{[v]} :  Out^0(\AG) \to \prod Out(A_{\Jv})
\]
and a projection homomorphism
 \[
P= \prod P_{[v]} :  Out^0(\AG) \to \prod Out(A_{\Lv}).
\]

\section{The kernel of the projection homomorphism}

In order to use the projection homomorphism $P$ to obtain information about $Out(\AG)$ we need to understand basic properties of its kernel.   We first consider the kernel of the restriction homomorphism $R$.  
\begin{theorem}\label{K}
The kernel $K$ of the homomorphism $R$ is a finitely generated free abelian group.  
\end{theorem}

\begin{proof}  If $\G_0$ consists of a single vertex $[v]$, then $\G = \Jv$ so the kernel is trivial.  So assume that there is more than one maximal equivalence class.

By definition, elements of $K$ are outer automorphisms.   We begin by choosing a canonical automorphism to represent each element of $K$.  First note that for any element $\phi$ of $K$ and any maximal $[v]$, we can choose a representative automorphism $\phi_v$ such that the restriction of $\phi_v$ to $A_{\Jv}$ is the \emph{identity} map.  This representative is unique up to conjugation by an element of the centralizer $C(\Jv)$.  In particular, if $[v]$ is non-abelian, then $\phi_v$ is unique.  

Suppose $V_{fr} \neq \emptyset$.  Choose a non-abelian class $[y]$.  Then for any element $\phi$ of $K$, we take $\phi_0=\phi_{y}$ as our canonical representative.  If $V_{fr} = \emptyset$, choose a pair of adjacent maximal classes $[y]$ and $[z]$.  Note that  $\phi_{y}=c(a)\phi_{z}$where $c(a)$ denotes conjugation by an element $a \in C(J_{y,z})$.  By Lemma \ref{adjacent-joins}, $a=rs$ for some $r \in A_{[y]}, s \in A_{\Ly}$. Set
\[
\phi_0 = c(r^{-1})\phi_y = c(s)\phi_z.
\]
 Then $\phi_0$ has the property that (i) it restricts to the identity on vertices of $\Jy$ and (ii) it acts on $\Jz$ as conjugation by an element of $A_{\Ly}$.  If $\phi_1$ is any other representative of $\phi$ with these two properties, then it differs from $\phi_0$ by conjugation by an element of $A_{[y]} \cap A_{\Ly} =\{1\}$.  Thus, $\phi_0$ is the unique such representative and we designate it as our canonical representative.  

The  properties which characterize canonical representatives are preserved under composition, thus the map $\phi \mapsto \phi_0$ defines a homomorphism of $K$ into $Aut(\AG)$.  For the remainder of the proof we view $K$ as a subgroup of the automorphism group by identifying $\phi$ with $\phi_0$.

To prove the theorem we will define a homomorphism $f$ of $K$ into a free abelian group and prove that $f$ is injective.    Recall that $V$ is the set of vertices of $\G$ so abelianizing gives a homomorphism $\AG \to  \Z^V$. Denote the abelianization of $g$ by $\bar g$.  For each maximal equivalence class $[v]$, define a homomorphism $f_v : K \to \Z^{V-[v]}$ as follows.  An element $\phi \in K$,  acts on $\Jv$ as conjugation by some $g \in \AG$.  The element $g$ is unique up to multiplication by an element of $C(\Jv) \subset A_{[v]}$, thus $\bar g$ determines a well defined element of $\Z^{V-[v]}$ (which by abuse of notation we will also denote $\bar g$).  

We claim that $f_v(\phi)=\bar g$ is a homomorphism.  For suppose $\rho$ is another element of $K$ and suppose $\rho$ acts on $\Jv$ as conjugation by $h$. Then $\phi \circ \rho$ acts on $\Jv$ as conjugation by $\phi(h) g$.  Since $\phi$ takes every generator to a conjugate of itself, it leaves the abelianization unchanged.  That is,  $\overline{(\phi(h))}= \bar h$, so $f_v(\phi \circ \rho)=\overline{ \phi(h) g}=\bar g + \bar h = f_v(\phi) + f_v(\rho)$. 

Now consider the product homomorphism $f=\prod f_v$ taken over all maximal equivalence classes $[v]$, 
\[
f : K \to \prod \Z^{V-[v]}. 
\]
To complete the proof, we will show that $f$ is injective.  Suppose $\phi$ lies in the kernel of $f$ and 
suppose $[v]$ and $[w]$ are adjacent maximal classes.  If $\phi$ acts on $\Jv$ as conjugation by $g_v$ and on $\Jw$ as conjugation by $g_w$, then $g_w^{-1}g_v$ lies in the centralizer of $J_{v,w}$, namely in the abelian group $Z([v]) \times Z([w]) \times Z(\Lv \cap \Lw)$.  Since $f_v(\phi)=\overline{g_v}=0$ and $f_w(\phi)=\overline{g_w}=0$, the exponent sum of any $u \in \Lv \cap \Lw$ is zero in both $g_v$ and $g_w$, hence also in $g_w^{-1}g_v$. It follows that  $g_w^{-1}g_v$  lies in $Z([v]) \times Z([w])$.  

Now consider an edge path in $\G_0$ from the base vertex $[y]$ to an arbitrary vertex $[v]$,
\[
[y]=[v_0], [v_1], \dots ,[v_n]=[v],
\]
and suppose $\phi$ acts on $J_{[v_i]}$ as conjugation by $g_i \in \AG$.  Our canonical representative $\phi$ was chosen so that $g_0=1$.  By the discussion above, $g_n=(g_0^{-1}g_1)(g_1^{-1}g_2)\cdots (g_{n-1}^{-1}g_n)$ is of the form $g_n=a_0a_1a_2 \dots a_n$, where $a_i \in Z([v_i])$.  Since $\phi$ lies in the kernel of $f$,  $f_{[v]}(\phi)=0$, so all of the $a_i$'s, except possibly the last one $a_n$, are trivial.  Thus, $\phi$  acts on $\Jv$ as conjugation by an element of its center, $Z([v])$, i.e., $\phi$ acts trivially on $\Jv$.  Since $[v]$ was arbitrary, we conclude that $\phi$ is the identity automorphism.
\end{proof}

In the case that $\G$ is connected and has no triangles, a stronger version of this theorem is proved in the authors' previous paper  with J.~Crisp \cite{ChCrVo}.  In that case, we give the exact rank of $K$ and determine an explicit set of partial conjugations which generate $K$.

Now let $K_P$ be the kernel  of the projection homomorphism
\[
P= \prod P_{[v]} :  Out^0(\AG) \to \prod Out(A_{\Lv}).
\]
Define a vertex $v$ to be {\it leaf-like} if
\begin{enumerate}
\item  $\Lv$ contains a unique maximal class $[w]$, and
\item  $[v] < [w]$.
\end{enumerate}
(An easy exercise shows that if $\G$ has no triangles, then a vertex is leaf-like if and only if it is a leaf, i.e. has valence 1.)
Since $d(v,w)=1$, it follows from Lemma \ref{chain} that $w$ must belong to $V_{ab}$. 
Since $[v] < [w]$, there is a transvection $t(v,w)$ taking $v \mapsto vw$.  We will call this a {\it leaf transvection}.

\begin{theorem}\label{K_P}  Assume $\G_0$ has at least two vertices. Then the kernel $K_P$ is a free abelian group generated by $K$ and the set of leaf transvections.
\end{theorem}

\begin{proof}  From the definition of ``leaf-like", it is easy to see that the leaf transvections $t(v,w)$ generate a free abelian group contained in the kernel of $K_P$ and that this subgroup has trivial intersection with $K$.  Thus, it suffices to show that leaf transvections commute with $K$ and that $K_P$ is generated by $K$ and the leaf transvections.  

Since every element of $K$ sends each generator to a conjugate of itself, Theorem 2.2 of \cite{Lau95} says that $K$ lies in the subgroup generated by  partial conjugations.  
 We claim that $t(v,w)$ commutes with all partial conjugations and hence with all of $K$.  The only case in which this could fail is the case of a partial conjugation by $u$ where $v$ and $w$ lie in different components of $\G \backslash st(u)$.  But this is impossible since $v$ and $w$ are connected by an edge.  If the edge lies in $st(u)$, then so do $v$ and $w$, so $u$ commutes with both of them.

It remains to show that $K$ and the leaf transvections generate all of $K_P$.  Suppose $\phi \in K_P$.  Let $[w]$ be maximal and $\phi_w$ be a representative automorphism preserving $A_{[w]}$ and $A_{\Jw}$. Since inner automorphisms by elements of $\Lw$ also preserve these subgroups, we may assume that $\phi_w$ projects to the identity automorphism on $A_{\Lw}$, so that 
 for any $v \in \Lw$, $\phi_w(v)=vg$ for some $g \in A_{[w]}$.  If  $[u]$ is another maximal equivalence class with $v \in \Lu$, then we also have $\phi_u(v)=vh$ for some $h \in A_{[u]}$.  However, $\phi_u$ and $\phi_v$ differ by an inner automorphism.  Since no (non-trivial) element of $A_{[w]}$ is conjugate to an element of $A_{[u]}$, this is impossible unless $g=h=1$.  

So suppose that $g \neq 1$ and $[w]$ is the unique maximal equivalence class with $v \in \Lw$  (or equivalently, $[w]$ is the unique maximal class in $\Lv$).   We claim that $C(v) \subset C(w)$.
In Theorem 1.2 of  \cite{Lau95}, Laurence gives a formula for the centralizer of a cyclically reduced element $x \in \AG$.  An easy corollary of his formula shows that if $x$ is a product of two commuting elements, $x=x_1x_2$, with disjoint support (i.e. $x_i \in A_{\Theta_i}$, with $\Theta_1$ and $\Theta_2$ disjoint, commuting set of vertices), then $C(x)=C(x_1) \cap C(x_2)$.  In particular, we have $C(vg)=C(v) \cap C(g) \subset A_{\Jv}$.  Applying the automorphism $\phi_w^{-1}$, then gives $C(v) \subset  A_{\Jv}$. 

On the other hand, we also have $A_{[w]} \subset C(v)$, so applying $\phi_w$ gives
$A_{[w]} \subset C(vg) \subset C(g)$.  The centralizer of an element in a non-abeliean free group cannot contain the entire free group, so we conclude that $w$ must lie in $V_{ab}$, and it follows that $C(g) = A_{\Jv} = C(w)$.  Combining these we get $C(v) \subset C(w)$, or equivalently, $st(v) \subset st(w)$.  Thus $[v] < [w]$ and $[v]$ is leaf-like.

In light of the discussion above, we can compose $\phi$ with an element of the leaf transvection group to obtain an outer automorphism $\tilde\phi$ such that for every $[w] \in \G_0$, there is a representative $\tilde\phi_w$ which preserves $A_{\Jw}$ and acts as the identity on $\Lw$.  If $[v]$ is adjacent to $[w]$  in $\G_0$, then $[w]$ is contained in $\Lv$ so there is another representative $\tilde\phi_v$ which preserves $A_{\Jv}$ and acts as the identity on $[w]$. (Here we are using the hypothesis that $\G_0$ contains at least two vertices.)  These two representatives differ by an inner automorphism which preserves $A_{J_{v,w}}$, so  $\tilde\phi_w$ acts on $[w]$ as conjugation by an element $g \in N(J_{v,w}) \subset A_{\Jw}$.   Since $[w]$ commutes with $\Lw$, we may assume that $g$ lies in $A_{[w]}$.  We conclude that the restriction of $\tilde\phi_w$ to $A_{\Jw}$ is conjugation by an element of $A_{[w]}$, an inner automorphism.   This holds for all maximal $[w]$, hence $\tilde\phi$ lies in $K$.   
\end{proof}

It remains to consider the special case in which $\G_0$ consists of a single vertex.

\begin{lemma}
The following are equivalent
\begin{enumerate}
\item $\G_0$ consists of a single vertex $[v]$.
\item $\G = \Jv$ for some $v \in V_{ab}$.
\item The center of $\AG$ is non-trivial.
\end{enumerate}
\end{lemma}

\begin{proof}  (ii) $\Rightarrow$ (iii) follows from Lemma~\ref{centers}.  (iii) $\Rightarrow$ (i) is clear since if $v \in \G\p$, then $st(v)=\G$ so $[v]$ is necessarily the unique maximal class.  For (i) $\Rightarrow$ (ii), note that for any $u \in \Lv$, $[v] \subset lk(u)$.  If $[v]$ is non-abelian, then $[v] \nsubseteq st(v)$,  so $[u] \nleq [v]$.  Hence if $[v]$ is the unique vertex in $\G_0$, it must be abelian, and it follows from Lemma~\ref{G_0} that $\G=\Jv$.
\end{proof}

\begin{proposition} \label{Out(G_0)}
If $\G_0$ consists of a single vertex $[v]$, then 
\[
Out(\AG)= Tr \rtimes (GL(A_{[v]}) \times Out(A_{\Lv}))
\]
where $Tr$ is the free abelian group generated by the leaf transvections $t(u,w)$ with $u \in \Lv$ and  $w \in [v]$.  
\end{proposition}

\begin{proof}  Any outer automorphism of $\AG$ preserves the center $A_{[v]}$ and projects via $P_{[v]}$ to $Out(A_{\Lv})$.  This gives a homomorphism
\[
Out(\AG) \to GL(A_{[v]}) \times Out(A_{\Lv})
\]
which is clearly split surjective.  It is easy to check that the kernel of this homomorphism is the group generated by leaf transvections. 
\end{proof}


\section{Virtual cohomological dimension}

If $\Gamma$ is discrete or is a complete graph, then $\AG$ is free or free abelian and $Out(\AG)$ is known to have torsion-free subgroups of finite index, so that the {\it virtual cohomological dimension} (vcd) of $Out(\AG)$ is defined.   In this section we prove that the same is true for arbitrary $\Gamma$, and we show furthermore that the vcd of $Out(\AG)$ is always finite.  

We define the {\it dimension of $\AG$} to be the maximal rank of a free abelian subgroup of $\AG$. This is determined by the number of vertices in the largest complete subgraph of $\G$, and we also call this the \emph{dimension} of $\G$.  Thus $\G$ has dimension 1 if and only if $\AG$ is free.    If $\G$ has dimension $n>1$, then links of vertices in $\G$ have dimension at most $n-1$.  Since the image of $P$ lies in the product of $Out(A_{\Lv})$, it is natural to try to use $P$ together with inductive arguments to prove properties of $Out(\AG)$.  Even for connected $\G$, the subgraphs $\Lv$ are not in general connected, so we must consider the disconnected case.  In dimension 2 the links always generate free groups and the theory of $Out(F_n)$ comes into play;  to deal with the general case  we must also appeal to the result of Guirardel and Levitt \cite{GuiLev07} stated below.  

 If $\G$ has $j$ components consisting of a single point and $k$  components, $\G_1,  \dots ,\G_k$ consisting of more than one point, then $\AG$ splits as a free product 
\[
\AG = F_j \ast A_1 \ast \dots \ast A_k
\]
where $F_j$ is a free group and  $A_i$ is the right-angled Artin group associated to $\G_i$.

\begin{theorem}[Guirardel-Levitt]\label{GL}   Suppose $G$ is a group which decomposes as a free product $G = G_1 \ast \dots \ast G_n$ with at least one factor non-free.  Assume that $G_i$ and $G_i/Z(G_i)$ are torsion-free (respectively, have finite virtual cohomological dimension) for all $i$.  If the outer automorphism groups $Out(G_i)$  are  virtually torsion-free (respectively, have finite virtual cohomological dimension) for each factor $G_i$, then the same is true for $Out(G)$.
\end{theorem}

We are now in a position to prove our theorem.

\begin{theorem}  For any finite simplicial graph $\G$, the group $Out(\AG)$ is virtually torsion-free and has finite virtual cohomological dimension. 
\end{theorem}

\begin{proof}  We proceed by induction on the dimension of $\G$.  If $\G$ has dimension $1$, then $\AG$ is free and the theorem follows from \cite{CulVog86}. 

Now suppose that $\G$ is connected and has dimension $\geq 2$.  First assume $\G_0$ has more than one vertex and consider the homomorphism $P$ defined in Section~\ref{projections}.   By induction, for every maximal $[v]$, $Out(A_{\Lv})$  is virtually torsion-free and has finite vcd, so the same holds for the image of $P$.  By Theorem~\ref{K_P}, the kernel $K_P$ is a finitely generated free abelian group, so in particular, it is torsion-free and has finite cohomological dimension. 
It now follows immediately that $Out^0(\AG)$ is virtually torsion-free and by the Serre spectral sequence, it has finite vcd.  Since $Out^0(\AG)$ is finite index in $Out(\AG)$, the same holds for the larger group.  

If $\G_0$ consists of a unique vertex $[v]$, we use Proposition~\ref{Out(G_0)}.  By induction, $Out(A_{\Lv})$  is virtually torsion-free and has finite vcd, and
the same is classically true for $Out(A_{[v]})=GL(A_{[v]})$.  Since the transvection group $Tr$ is free abelian, the theorem follows as above.

Finally, applying Theorem~\ref{GL}, these results extend to $n$-dimensional graphs $\G$ with more than one component.  This completes the induction.
\end{proof}

In \cite{ChCrVo}, the authors and J.~Crisp studied the case in which $\G$ is connected and 2-dimensional.  They obtained explicit upper and lower bounds on the vcd of $Out(\AG)$ and constructed a contractible ``outer space" with a proper $Out(\AG)$ action.  In a forthcoming paper with K.-U. Bux, the authors determine the exact vcd for many 2-dimensional right-angled Artin groups, in particular those whose defining graph is a tree  \cite{BuChVo}.

It was also shown in \cite{ChCrVo} that for connected, 2-dimensional $\G$, $Out(\AG)$ satisfies the Tits alternative.  One would like to do an inductive argument as above to show that this holds for all $\G$.  However, in this case, we do not have the analogue of Theorem \cite{GuiLev07} to pass from the connected to the disconnected case. 

\medskip
{\it Acknowledgements}:  The authors would like to thank Eddy Godelle for useful comments.


\def\cprime{$\prime$}

\affiliationone{ 
   Ruth Charney\\
        Mathematics Department\\
        Brandeis University\\
        Waltham, MA 02454-9110\\
  U.S.A.
   \email{charney@brandeis.edu}}
\affiliationtwo{ 
   Karen Vogtmann\\
   Mathematics Department\\
   Cornell University\\
   Ithaca, NY 14853-4201\\
   U.S.A.
   \email{vogtmann@math.cornell.edu}}

\end{document}